\documentclass[twoside,11pt,reqno]{amsart}

\usepackage{upref,amsxtra,amssymb,amsmath, amscd}
\usepackage{mathrsfs}  
\usepackage{varioref}
\usepackage{verbatim}
\usepackage{epsfig}
\usepackage{color}
\usepackage{url}
  
\usepackage{epic,eepic,eucal}

\usepackage{enumerate}

\def\Cc{            \mathcal{C}}
\def\Ee{            \mathcal{E}}
\def\Bb{            \mathcal{B}}

\def\Tt{            \mathcal{T}}

\def\Or{          \mathcal O}

\def\mS{          \mathcal S}

\def\a{         \alpha}

\newcommand{\dist}{\operatorname{dist}}
\newcommand{\NN}{{\mathbb N}}
\newcommand{\RR}{{\mathbb R}}

\newcommand{\TT}{{\mathbb T}}

\newcommand{\QQ}{{\mathbb Q}}

\newcommand{\Pp}{{\mathcal P}}

\newtheorem{theo}{\sc Theorem}[section]
\newtheorem{prop}[theo]{\sc Proposition}

\newtheorem{lemm}[theo]{\sc Lemma}
\newtheorem{coro}[theo]{\sc Corollary}

\theoremstyle{definition}

\theoremstyle{remark}

\newtheorem{rema}[theo]{\sc Remark}

\numberwithin{equation}{section}

\usepackage{hyperref}
\usepackage{graphicx} 
\usepackage{booktabs} 
\usepackage{wrapfig} 
\usepackage[labelfont=bf]{caption} 
\usepackage[top=0.9in,bottom=0.9in,left=1in,right=1in]{geometry} 

\begin{document}
\title{Simultaneous dense and nondense orbits for toral diffeomorphisms}
\author{Jimmy Tseng}
\address{School of Mathematics, University of Bristol, University Walk, Bristol, BS8 1TW UK}
\email{j.tseng@bristol.ac.uk}

\thanks{The author acknowledges the research leading to these results has received funding from the European Research Council under the European Union's Seventh Framework Programme (FP/2007-2013) / ERC Grant Agreement n. 291147.}

\begin{abstract} We show that, for pairs of hyperbolic toral automorphisms on the $2$-torus, the points with dense forward orbits under one map and nondense forward orbits under the other is a dense, uncountable set.  The pair of maps can be noncommuting.  We also show the same for pairs of $C^2$-Anosov diffeomorphisms on the $2$-torus.  (The pairs must satisfy slight constraints.)  Our main tools are the Baire Category theorem and a geometric construction that allows us to give a geometric characterization of the fractal that is the set of points with forward orbits that miss a certain open set.
\end{abstract}
\maketitle
\section{Introduction}\label{secIntro}  Given a dynamical system $f: X \rightarrow X$ on a set $X$ with a topology, we say that a point has \textit{dense forward orbit} if its forward orbit closure equals $X$ and \textit{nondense forward orbit} otherwise.  Let the set of points with dense forward orbits be called the \textit{dense set for $f$} and denoted by $D(f)$ and with nondense forward orbits be called the \textit{nondense set for $f$} and denoted by $ND(f)$.  Let $\tilde{f}:X \rightarrow X$ be another dynamical system on the same phase space.  How large is the set $D(f) \cap ND(\tilde{f})$ is a natural question to ask.  It was first asked by V.~Bergelson, M.~Einsiedler, and the author in~\cite{BET}, in which complete orbits (as well as forward orbits) were considered for commuting pairs of maps on the torus and on certain compact homogeneous spaces.  The technique in the present paper is completely different, but applies to pairs of noncommuting maps as well.  The results in this paper were proven before the results in~\cite{BET} and complement them.\footnote{After the preprint version of this paper, which was entitled {\em Simultaneous dense and nondense orbits for toral automorphisms}, appeared on arXiv, the paper~\cite{LM} appeared on arXiv showing the case for noncommuting linear maps on the $d$-torus using completely different methods from that in this paper.  Also since the preprint version of this paper appeared, R.~Shi and the author have proved in~\cite{ST} results for dense and nondense orbits on noncompact spaces.}  The technique in this paper, which involves the construction of a certain fractal with the help of periodic points, could be useful for other considerations.

\subsection{Statement of results}  Let $g: \TT^2 \rightarrow \TT^2$ be a $C^2$-Anosov diffeomorphism.  For a point $x \in \TT^2$, let $E_g^+(x)$ denote the {\em unstable manifold for $g$ through $x$} and $E_g^-(x)$, the {\em stable manifold for $g$ through $x$}.  The collection of stable manifolds $\{E^-_g(x)\}_{x \in \TT^2}$ form the \textit{stable foliation} and an element in this collection is referred to as a \textit{leaf}.  Likewise, for unstable manifolds.  (See Section~\ref{secProofthmForDenNonDenAnosov} for more details.)  Let $E_g^+:= E_g^+(\boldsymbol{0})$ and $E_g^-:= E_g^-(\boldsymbol{0})$.  

Let $T:\TT^2 \rightarrow \TT^2$ be a hyperbolic toral automorphism.  Our first main result is the following theorem and corollary.

\begin{theo}\label{thmForDenNonDen}Let $T, S:\TT^2 \rightarrow \TT^2$ be hyperbolic toral automorphisms.  If $\dim(E_T^- \cap E_S^-) = 0$, then $ND(T) \cap D(S)$ is a dense, uncountable set.
\end{theo}

\begin{coro}\label{coroForDenNonDen}Let $T$ be a hyperbolic toral automorphism of $\TT^2$ and $\{S_k\}$ be the family of all hyperbolic toral automorphisms of $\TT^2$ such that $\dim(E_T^- \cap E_{S_k}^-) = 0$.  Then $ND(T) \cap \bigcap_k D(S_k)$ is a dense, uncountable set.
\end{coro}

For $C^2$-Anosov diffeomorphisms of the $2$-torus, it is well-known that the stable foliations are $C^1$ (see~\cite[Corollary~4]{HP} or~\cite[Corollary~19.1.11 and its remark]{HKIntroMD}) and the notion of transverse intersections is well-defined.  Our second main result is the following theorem and corollary and is the first time nonlinear maps are considered for simultaneous dense and nondense orbits.  

\begin{theo}\label{thmForDenNonDenAnosov}  Let $\widetilde{T}, \widetilde{S}:\TT^2 \rightarrow \TT^2$ be $C^2$-Anosov diffeomorphisms.  If the intersection points of every leaf in the stable foliation for $\widetilde{T}$ with every leaf in the stable foliation for $\widetilde{S}$ are transverse intersections, then $ND(\widetilde{T}) \cap D(\widetilde{S})$ is a dense, uncountable set.
\end{theo}

\begin{coro}\label{coroForDenNonDenAnosov}  Let $\widetilde{T}$ be a $C^2$-Anosov diffeomorphism of $\TT^2$ and $\{\widetilde{S}_k\}$ be a countably infinite (or finite) family of $C^2$-Anosov diffeomorphisms of $\TT^2$ such that the intersection points of every leaf in the stable foliation for $\widetilde{T}$ with every leaf in the stable foliation for $\widetilde{S}_k$ are transverse intersections.  Then $ND(\widetilde{T}) \cap \bigcap_k D(\widetilde{S}_k)$ is a dense, uncountable set.
\end{coro}

\noindent Note that the nondense sets for these maps are known to be Lebesgue null sets of full Hausdorff dimension~\cite{BFK, T4, Ur}.

Finally, Lemma~\ref{lemmForwardFracContainsContrSpace}, which gives a geometric characterization of the fractal that is the set of points with forward orbits that miss the open set constructed in Section~\ref{secConFracforT}, may be of independent interest.  This geometric characterization also holds in (and is very important for) the case of $C^2$-Anosov diffeomorphisms of the torus (see Lemma~\ref{lemmForwardFracContainsContrSpaceAnosov} and the proof of Theorem~\ref{thmForDenNonDenAnosov} in Section~\ref{secProofthmForDenNonDenAnosov}).

\subsection{Idea of proof and outline of paper}  The idea of the proof of Theorem~\ref{thmForDenNonDen} is as follows.  Using periodic points, we construct an open parallelogram $B$ and its iterates under $T$.  The dynamics gives us a type of geometric rigidity (Proposition~\ref{propNoProperOverlaps}) for these iterates and, in particular, preclude proper overlaps (Section~\ref{subsecPropOverlaps}).  This geometric rigidity allows us to give a geometric characterization (Lemma~\ref{lemmForwardFracContainsContrSpace}) of the fractal that encodes missing $B$ for the mapping $T$, a characterization that is strong enough to allow us to consider incidence geometry and robust enough to work under topological conjugacy.  Applying the Baire category theorem and recursively repeating this for a shrinking family of open parallelograms yields the desired result.  Sections~\ref{secConFracforT} and~\ref{secGeoFracforT} are devoted to Theorem~\ref{thmForDenNonDen} and its corollary.  The proof of Theorem~\ref{thmForDenNonDen} is in Section~\ref{secProofthmForDenNonDen}.

The proof Theorem~\ref{thmForDenNonDenAnosov}, which is found in Section~\ref{secProofthmForDenNonDenAnosov}, is a corollary of the proof of Theorem~\ref{thmForDenNonDen} and the global classification of Anosov diffeomorphisms on tori.  The key ingredient of the proof of Theorem~\ref{thmForDenNonDenAnosov} is the robustness of our geometric technique.

\subsubsection*{Acknowledgements}  I thank V.~Bergelson for suggesting the problem and discussions and the referee for useful comments on the exposition.

\section{Constructing the fractal for $T$}\label{secConFracforT}
Recall that $T: \TT^2 \rightarrow \TT^2$ is a hyperbolic toral automorphism.  Let \[E^+:= E^+_T \quad E^+(z) := E^+_T(z) \quad E^-:=E^-_T \quad E^-(z):=E^-_T(z),\] where $z \in \TT^2$.  Given a set $B \in \TT^2$, define the fractal \[F:=F_{T}(B) := \TT^2 \backslash \cup_{n=0}^\infty T^{-n}(B).\]  We refer to the elements of the collection $\{T^{-n}(B)\}_{n=0}^\infty$ as \textit{tubes} (or, more precisely, {\it $T^{-1}$-tubes for $B$}).  For two tubes $T^{-m}(B)$ and $T^{-n}(B)$ where $m<n$, we call $T^{-m}(B)$ the \textit{old tube} and $T^{-n}(B)$ the \textit{new tube}.

We will construct a small open parallelogram $B$ with one vertex at the origin $\boldsymbol{0}$ of $\TT^2$.  Our construction takes place inside a small-enough neighborhood of $\boldsymbol{0}$ so that locally we are on $\RR^2$.  Pick a small open parallelogram $B'$ with one vertex at $\boldsymbol{0}$ whose two edges $B'^-$ and $B'^+$ are small closed segments of $E^-$ and $E^+$, respectively.  Pick a rational point $y_0 \in B'$ and let $\{y_0, \cdots, y_{N-1}\} \subset \TT^2$ be the orbit of $y_0$ under $T$.

To construct the parallelogram $B$, we must correctly choose an iterate. Let $z \in \TT^2$ and $\a>0$.  Let \[P(\a, z) \subset E^-(z)\] denote the closed ball of $E^-(z)$ around $z$ of radius $\a$ and, similarly, \[Q(\a, z) \subset E^+(z)\] denote the closed ball of $E^+(z)$.  For each $y_p$, there exists a unique smallest $\a_p>0$ such that $P(\a_p, y_p)$ meets $B'^+$.  Let \[\Pp(\a) := \bigg{\{}P(\a, y_p)\bigg{\}}_{p=0}^{N-1}.\]

Let $0 \leq q<N$ be an index such that $\a_q$ is minimal in the set $\{\a_0, \cdots, \a_{N-1}\}$.  Note that since $y_0$ is chosen in the small enough open parallelogram $B'$ and $\a_q \leq \a_0$, each element of $\Pp(\a_q)$ can only intersect $B'^+$ at exactly one point, namely one of its endpoints.  Consider the following cases.

\subsection{The index $q$ is unique}  Thus $P(\a_q, y_q)$ meets $B'^+$ and is the only element of $\Pp(\a_q)$ to do so.  Call this intersection point $x$.

\subsubsection{The segment $P(\a_q, y_q)$ meets $B'$.}\label{susubsecUniMeetsInt}  Recall that $B'$ is open and thus does not contain $B'^+$.  We claim that $x \neq \boldsymbol{0}$. If not, then $P(\a_q, y_q) \subset E^-$, which implies that the periodic point $T^{\ell}(y_0) \rightarrow \boldsymbol{0}$ as $\ell \rightarrow \infty$.  Thus, the periodic point $y_0$ must be $\boldsymbol{0}$.  This contradicts the fact that $y_0$ is chosen in the open set $B'$ and shows our claim.  Let $\Ee \subset B'^+$ denote the closed segment between $\boldsymbol{0}$ and $x$.

For every $p \neq q$, there is a positive minimum distance between $P(\a_q, y_p)$ and  $B'^+$.  Consequently, there exists an open parallelogram $B \subset B'$ with $\Ee$ as an edge such that $B$ does not meet any $P(\a_q, y_p)$ (including $p = q$ since $B$ is open).  Finally, let  \[B_{\boldsymbol{0}}:= \overline{B} \cap B'^- \quad B_x := \overline{B} \cap P(\a_q, y_q).\]

\subsubsection{The segment $P(\a_q, y_q)$ does not meet $B'$.}  Picking an $\a$ slightly bigger than $\a_q$ will result in $P(\a, y_q)$ meeting $B'$, but $P(\a, y_p)$ not meeting $B'^+$ for any $p \neq q$.  Choose $B$ and denote $B_{\boldsymbol{0}}$ and $B_x$ as in Section~\ref{susubsecUniMeetsInt}.  

\subsection{The index $q$ is not unique}  Let $0\leq q_0, \cdots, q_k<N$ be all the indices such that $\a_q:= \a_{q_0} = \cdots = \a_{q_k}$.  Thus every element of $\Pp_0(\a_q):=\{P(\a_q, y_{q_i})\}_{i=0}^k$ meets $B'^+$ and are the only elements of $\Pp(\a_q)$ to do so.  

\subsubsection{At least one element of $\Pp_0(\a_q)$ meets $B'$}\label{subsubsecMultMeetsInt}  For each element of $\Pp_0(\a_q)$ that meets $B'$, there exists a unique intersection point with $B'^+$, and, since $B'$ is chosen small enough, there exists exactly one such intersection point $x$ that is nearest to $\boldsymbol{0}$.  For exactly the same reason as in Section~\ref{susubsecUniMeetsInt}, we have that $x \neq \boldsymbol{0}$.  To this $x$ corresponds a unique element of $\Pp_0(\a_q)$.  Let $\Ee \subset B'^+$ denote the closed segment between $\boldsymbol{0}$ and $x$.

Now consider each remaining element $P$ of $\Pp_0(\a_q)$.  Recall that $\a_q$ is so small that each $P$ can only intersect $B'^+$ in exactly one point.  Hence the other endpoint of $P$ does not intersect $B'^+$.

For each element of $\Pp(\a_q) \backslash \Pp_0(\a_q)$, no points intersect $B'^+$.  Consequently, there exists an open parallelogram $B \subset B'$ with $\Ee$ as an edge such that $B$ does not meet any element of $\Pp(\a_q)$.  Finally, let  \[B_{\boldsymbol{0}}:= \overline{B} \cap B'^- \quad B_x := \textrm{ the edge of } B \textrm{ parallel to } B_{\boldsymbol{0}}.\]

\subsubsection{No element of $\Pp_0(\a_q)$meet $B'$}  Picking an $\a$ slightly bigger than $\a_q$ will result in every element of $\{P(\a, y_{q_i})\}_{i=0}^k$ meeting $B'$, but $P(\a, y_p)$ not meeting $B'^+$ for any $p \notin \{q_0, \cdots, q_k\}$.  Choose $B$ and denote $B_{\boldsymbol{0}}$ and $B_x$ as in Section~\ref{subsubsecMultMeetsInt}.  This concludes the construction of the open parallelogram $B$.  Note that, in the construction of $B$, a unique element of $\Pp(\a_q)$ is chosen.

\subsection{Proper overlaps}\label{subsecPropOverlaps}  The notion of proper overlaps, to be defined towards the end of this section, is local.  However, we show below that proper overlaps cannot occur anywhere; thus a global condition on the family of tubes is obtained.

Recall the definitions of $B_{\boldsymbol{0}}$, $B_x$, $\Ee$, and $Q(\a, z)$.  Let \[\Cc:= B_{\boldsymbol{0}} \quad \Cc' := B_x \quad \Ee':= \textrm{ the edge of } B \textrm{ parallel to } \Ee\] and \[\Cc_m := T^{-m}(\Cc) \quad \Cc'_m := T^{-m}(\Cc')\quad \Ee_m := T^{-m}(\Ee) \quad \Ee'_m:= T^{-m}(\Ee').\] The edges denoted by $\Cc$ are referred to as \textit{contracting} and by $\Ee$ as \textit{expanding}.

Consider an old tube $T^{-m}(B)$ and a new tube $T^{-n}(B)$.  Let $y \in T^{-m}(B) \cap T^{-n}(B)$.  Since the intersection of the two tubes is an open set, there exists an $\a > 0$ such that $Q(\a,y)$ is contained in this intersection.  Thus, there exists unique minimal $\tilde{\a}:= \a_m, \a'_m, \a_n, \a'_n>0$ such that $Q(\tilde{\a},y)$ meets $\Cc_m$, $\Cc'_m$, $\Cc_n$, $\Cc'_n$ in the unique points $x_m$, $x'_m$, $x_n$, $x'_n$, respectively.  These intersection points will only coincide when they lie on the same edge.  (As an aside, note that the primed contracting edges cannot coincide with the non-primed contracting edges; see Lemma~\ref{lemmDisjointpieces} for a proof.)  The \textit{slice $\mS_m$ of the tube $T^{-m}(B)$ through $y$} is the open segment between $x_m$ and $x'_m$ and the \textit{closed slice $\overline{\mS}_m$} is the closed segment.  These segments are contained in $E^+(y) = E^+(x_m) = E^+(x_m')$.  The points $x_m$ and $x'_m$ are called the \textit{vertices} of the slice $\mS_m$.

\begin{lemm}\label{lemmSlicesAreTranslates} The slice $\overline{\mS}_m$ is the translation of $\Ee_m$ by $x_m$.
\end{lemm}
\begin{proof}
Lift, via the natural projection $\RR^2 \rightarrow \TT^2$, the closed tube $\Tt:=\overline{T^{-m}(B)}$ to $\RR^2$.  We obtain an infinite family of disjoint parallelograms.  Exactly one of these $\widetilde{\Tt}$ has the origin of $\RR^2$ as vertex.  Let $\widetilde{\Ee_m}$ denote the unique lift of $\Ee_m$, $\widetilde{x}_m$ denote the unique lift of $x_m$, and $\widetilde{\mS}_m$ denote the unique lift of $\overline{\mS}_m$  to $\widetilde{\Tt}$.  It follows that $\widetilde{\mS}_m$ is the translation of $\widetilde{\Ee_m}$ by $\widetilde{x}_m$ (i.e. $\widetilde{\mS}_m = \{z + \widetilde{x}_m : z \in \widetilde{\Ee_m}\}$).  Since translation on $\TT^2$ is the mapping that makes translation on $\RR^2$ commute with the natural projection map, the result follows.  \end{proof}

The proof is simplified if we assume that the eigenvalues of $T$ are positive real numbers.  The general case will follow easily from this (see the proof of Theorem~\ref{thmForDenNonDen} in Section~\ref{secBaireCatArg}).  Recall that the old and new tubes intersect at a point $y$ and that $\mS_m$ is the slice of the old tube through $y$.  Let $\mS_n$ be the slice of the new tube through $y$.   

\begin{lemm}\label{lemmSliceofOldandNewTubes} Let the eigenvalues of $T$ be positive real numbers.  For any closed slice $\overline{\mS}_m$ of a tube $T^{-m}(B)$, $T^j (\overline{\mS}_m)$ is the closed slice $\overline{\mS}_{j-m}$ of the tube $T^{j-m}(B)$ with vertices $ x_{j-m} = T^j(x_m)$ and $x_{j-m}'=T^j(x_m')$.
\end{lemm}

\begin{proof}
 Since $T^j E^+(x_m) = E^+(T^j x_m)$ and $\overline{\mS}_m \subset E^+(x_m)$, the conclusion follows by the definition of closed slice.
\end{proof}

\begin{lemm}\label{lemmContainsSlice} Let the eigenvalues of $T$ be positive real numbers.  If both $x_n$ and $x_n'$ are in $\mS_m$, then $\overline{\mS}_n \subset \mS_m$.
\end{lemm}

\begin{proof}

The segments $\mS_m$ and $\overline{\mS}_n$ are convex. \end{proof}

\begin{lemm}\label{lemmNewTubesAreSmaller} Let the eigenvalues of $T$ be positive real numbers.  At least one vertex of $\mS_m$ is not contained in $\overline{\mS}_n$.
\end{lemm}

\begin{proof}
By Lemma~\ref{lemmSliceofOldandNewTubes}, $T^{-(n-m)} (\overline{\mS}_m)$ is a closed slice through the new tube.  By Lemma~\ref{lemmSlicesAreTranslates}, $\overline{\mS}_n$ is a translate of $T^{-(n-m)} (\overline{\mS}_m)$.  Since translation on $\TT^2$ is an isometry, $\overline{\mS}_n$ has the same length as $T^{-(n-m)} (\overline{\mS}_m)$.  But, the length of $\overline{\mS}_m$ is strictly greater than that of $T^{-(n-m)} (\overline{\mS}_m)$ because $n - m > 0$ and $\overline{\mS}_m \subset E^+(x_m)$.  The result is now immediate.
\end{proof}

We now define the notion of a proper overlap for the old and new tubes.  We say that the two tubes \textit{overlap properly at $y$} if $\mS_n \cap \{x_m, x_m'\} \neq \emptyset$.  Since the tubes are open sets, once the overlap property holds (or, respectively, does not hold) at $y$, then there is a small neighborhood of $y$ in the intersection of the two tubes for which the overlap property holds (or, respectively, does not hold). By convexity, if $\mS_n$ contains both $x_m$ and $x_m'$, then $\overline{\mS}_m \subset \mS_n$, a contradiction of Lemma~\ref{lemmNewTubesAreSmaller}.  Thus, when the tubes overlap properly, $\mS_n$ contains either $x_m$ or $x_m'$, but not both.

The following proposition provides a geometric understanding of how (open) tubes behave.

\begin{prop}\label{propNoProperOverlaps} Let the eigenvalues of $T$ be positive real numbers.  For the constructed parallelogram $B$, two distinct tubes do not properly overlap anywhere.
\end{prop}

\begin{proof}

Let the old tube be $T^{-m}(B)$ and the new tube be $T^{-n}(B)$.  Let $y \in \TT^2$.  If $y$ is not in the intersection of the two tubes, then the two tubes do not properly overlap at $y$.

Thus we need only consider when $y$ is in the intersection of the two tubes.  Assume that the two tubes overlap properly at $y$. Let $\mS_m$ be the slice of the old tube through $y$ and $\mS_n$ be the slice of the new tube through $y$.  Since the overlap is proper, $\mS_n$ contains either $x_m$ or $x_m'$, but not both.

\textit{Case:  $x_m \in \mS_n$.}  Applying $T^m$ to the old tube will map it into the original parallelogram $B$.  Also, the point $x_m$ will return to the contracting edge of $B$, namely the set $B_{\boldsymbol{0}}$.  For all $k \geq m$, $T^k(x_m) \in B_{\boldsymbol{0}}$ because $B_{\boldsymbol{0}} \subset E^-$.  Applying $T^n$ to the new tube will map it into the original parallelogram $B$.  Since $B$ is open, it does not contain any point in any of its edges and therefore $T^n(x_m) \notin B_{\boldsymbol{0}}$, a contradiction. 

\textit{Case:  $x_m' \in \mS_n$.}  Similar to the previous case, $T^m(x_m') \in B_x$.  As before, applying $T^n$ to the new tube will map it into $B$.  Therefore $T^n(x_m') \in B$.  Let $P(\a_q, y_q)$ be the unique element of $\Pp(\a_q)$ chosen in the construction of $B$.

Now, by construction, $P(\a_q, y_q)$ contains $B_x$ and no element of $\Pp(\a_q)$ meets $B$.  Since $T$ permutes the periodic points $\{y_\ell\}_{\ell =0}^{N-1}$ and is contracting on any $P(\a_q, y_\ell)$, we have that for each $P(\a_q, y_\ell)$ there exists a $P(\a_q, y_j)$ such that $T(P(\a_q, y_\ell)) \subset P(\a_q, y_j)$.  Thus, $T^i(x_m') \notin B$ for every $i \geq m$, which is a contradiction.
 \end{proof}

\section{The geometry of the fractal for $T$}\label{secGeoFracforT}

Recall that the fractal is the set $F:=F_T :=F_T(B):=\TT^2 \backslash \bigcup_{n=0}^\infty T^{-n} B$ where $B$ is the parallelogram constructed in Section~\ref{secConFracforT}.  Recall that $B$ is an open set and therefore its tubes are also open sets.

\begin{lemm}\label{lemmDisjointpieces} Let the eigenvalues of $T$ be positive real numbers.  The sets $\Cc_M, \Cc_M'$ and $\bigcup_{n=M+1}^\infty T^{-n} B$ are pairwise disjoint.
\end{lemm}
\begin{proof}

Let $P(\a_q, y_q)$ be the unique element of $\Pp(\a_q)$ chosen in the construction of $B$.  The stable manifolds $E^-$ and $E^-(y_q)$ are leaves in the stable foliation for $T$.  As distinct leaves of the same foliation are disjoint, the only way for them to intersect is if they coincide.  Recall that $\Cc_M \subset E^-$ and $\Cc_M' \subset E^-(y_q)$.  Thus, if we assume that $\Cc_M, \Cc_M'$ are not disjoint, then $y_q \in E^-$, but the proof of the claim in Section~\ref{susubsecUniMeetsInt} gives a contradiction.

Assume that $\Cc_M$ meets $\bigcup_{n=M+1}^\infty T^{-n} B$.  Then there exists $y \in \Cc_M \cap T^{-n} B$ for some $n > M$.  Call the tube corresponding to $M$ the old tube and $n$, the new tube.  Taking slices through the two tubes at $y$ shows that they overlap properly at a point $y'$ in the intersection of the two tubes.  (Recall that the both tubes are open, so $y'$ lies on the slice through $y$ but cannot be $y$, as $y$ lies on the edge of the old tube and not in the open intersection.) This contradicts Proposition~\ref{propNoProperOverlaps}.  The same proof suffices for $\Cc_M'$.
\end{proof}

The fractal $F$ is the set of points with forward orbits that miss $B$.  The following key lemma geometrically characterizes $F$ and is of independent interest.  First, it would be convenient to define a few terms.  We define a \textit{piece of space parallel to the stable manifold} to be a translate in $\TT^2$ of an (small) open segment contained in $E^-$.  Furthermore, we define a \textit{$T$-parallelogram} to be a parallelogram in $\TT^2$ with one pair of parallel edges parallel to $E^-$ (called the \textit{contracting edges}) and the other pair of parallel edges parallel to $E^+$ (called the \textit{expanding edges}).  An \textit{open} $T$-parallelogram is, in addition, an open set.  A \textit{slice} of a $T$-parallelogram is defined in the analogous way as the slice of a tube.  Also, we must now distinguish two types of slices for tubes:  interior and boundary. Given a tube $T^{-k}B$, the two distinct \textit{boundary slices} are $\Ee_k$ and $\Ee_k'$.  No iterate of the primed boundary slice can intersect an iterate of the unprimed boundary slice (follows by the analog of the proof of Lemma~\ref{lemmDisjointpieces}).  All other slices of the tube are \textit{interior slices.} 

We also define directions for the parallelogram $B$ and its tubes as follows.  Recall the two edges $B_{\boldsymbol{0}}$ and $\Ee$ of $B$ from Section~\ref{secConFracforT}.  Regarding these edges as vectors pointing away from $\boldsymbol{0}$ yields the \textit{positive contracting direction} which points in the direction of $B_{\boldsymbol{0}}$ and the \textit{positive expanding direction} which points in the direction of $\Ee$.  The \textit{negative} directions point in the respective opposite directions.  Since $T$ has positive eigenvalues, $T$ preserves these directions and hence these directions are defined in the same way for the tubes.  Thus, at any point $z$ in a $T^{-k}$-tube near the interior of $\Ee_k'$, we may move along $E^-(z)$ in the negative contracting direction and still remain inside the tube.  Likewise, at any point $z'$ in a $T^{-k}$-tube near the interior of $\Ee_k$, we may move along $E^-(z')$ in the positive contracting direction and still remain inside the tube.  In both cases, we may move almost as far as the length of $\Cc_k$.

Finally, we remark that, in the proof of the following lemma, it is important to note that the contracting edges of the oldest tube to meet an open set is guaranteed to become part of the fractal $F$ because there are no proper overlaps.  However, complete overlaps can occur, so even if the contracting edges of newer tubes meet this open set, they are not guaranteed to become part of the fractal, as they could lie in an older tube.  Consequently, the proof must proceed, as it does, by recursion on each iteration of a tube that meets the (relevant) open set.

\begin{lemm}\label{lemmForwardFracContainsContrSpace} Let the eigenvalues of $T$ be positive real numbers.  Any open subset of $F$ contains a piece of space parallel to the stable manifold. 
\end{lemm}
\begin{proof}
Let $U$ be an open $T$-parallelogram of small diameter containing a point $y \in F$.  In particular, we assume that diam$(U)$ is much smaller than the length of $\Cc_0$.  By ergodicity, the orbit of a point in $B$ meets $U$.  Therefore, some tube meets $U$.  Since $y \in F$, no tube can completely contain $U$.  Let $M \geq 0$ correspond to the first tube that meets $U$.  Since $U$ is not completely contained in this tube, $U$ meets the boundary of the tube.  There are two main cases.  

\bigskip\noindent
\textit{Case 1:  $U$ meets $\Cc_M \cup \Cc_M'$.}  

Let $z$ be a point in this intersection and assume that it is in $\Cc_M$.  Since $U$ is open, $U$ meets the interior of $\Cc_M$.  The intersection of $U$ and this interior is a piece of space parallel to the stable manifold. By Lemma~\ref{lemmDisjointpieces}, this interior will be in $F$.  The same proof suffices for $\Cc_M'$.  This concludes Case 1.

\bigskip\noindent
\textit{Case 2: $U$ does not meet $\Cc_M \cup \Cc_M'$.}

Since $U$ does not meet $\Cc_M \cup \Cc_M'$ and it cannot be contained inside the tube, it must meet the interior of $\Ee_M$ or the interior of $\Ee_M'$.  Since $T$ is an automorphism, $\Ee_M$ and $\Ee_M'$ cannot intersect.  Thus, if $U$ meets both the interior of $\Ee_M$ and the interior of $\Ee_M'$, then, by shrinking $U$ only along the contracting edges, we can have $U$ contain $y$ and the interior of either $\Ee_M$ or $\Ee_M'$, but not both. Since $y$ does not meet the tube, since the tubes and $U$ are open $T$-parallelograms, and since $U$ does not meet $\Cc_M \cup \Cc_M'$, the slice $\mS_U$ of $U$ through $y$ is does not meet the tube.  Moreover, since $U$ only meets the interior of $\Ee_M$ or $\Ee_M'$, the slice $\mS_U$ divides the parallelogram $U$ into two open halves:  one half $W$ intersecting the tube and the other $V$ not intersecting the tube.

\medskip\textit{Case 2A:  $U$ meets $\Ee_M$.}  First note that $V$ is in the negative contracting direction with respect to $W$.  Let $n > M$ correspond to the tube that next intersects the open $T$-parallelogram $V$.  If $V$ meets $\Cc_n \cup \Cc_n'$, then we are in Case 1 and thus finished.

Otherwise, $V$ must meet the new tube in the interior of $\Ee_n$ or the interior of $\Ee_n'$ or $V$ is completely contained in the new tube.  Now note that $\Ee_n \subset \Ee_M$ and thus $\Ee_n \cap U \subset W \cup \mS_U$ (of course, $\Ee_n \cap U$ may be empty), and, thus, $\Ee_n$ cannot meet $V$.  Since $V$ does not meet $\Cc_n \cup \Cc_n'$, then either $V$ meets the interior of $\Ee_n'$ or $V$ is completely contained in the new tube.  If the latter, we claim that the interior of $\Ee_n'$ also meets $U$.  If not, then starting at $\Ee_n$ we may move in the positive contracting direction and stay in the new tube until we meet $\Ee_n'$.  As we move along, we meet $V$ because the new tube meets $V$.  Since $U$ does not meet $\Ee_n'$ we move across $U$.  Thus the new tube completely contains $U$, a contradiction.    Therefore, we have proven our claim:  $U$ must meet the interior of $\Ee_n'$.

We note that the slice $U \cap \Ee_n'$ must not lie in the positive contracting direction with respect to $S_U$ because, otherwise, $U$ contains $y$.  Thus, either the boundary slices $\Ee_M$ and $\Ee_n'$ intersect in $U$ and therefore both contain the slice $\mS_U$ or they are distinct slices and therefore contain distinct slices of $U$ with only one of these slices allowed to possibly be $\mS_U$.  But since, as noted above, iterates of primed and unprimed boundary slices cannot intersect, the latter must hold.  Consequently, there exists an open parallelogram $\widetilde{U}$ contained in $U$ and containing $\mS_U$ and meeting either $\Ee_M$ or $\Ee_n'$, but not both.  The slice $\mS_U$ divides $\widetilde{U}$ into two halves:  one half $\widetilde{W}$ intersecting one of the tubes and the other $\widetilde{V}$ not intersecting either tube.  

Note that if there exists $M < j < n$ such that the $T^{-j}$-tube meets $U$, then the $T^{-j}$-tube must only meet $W \cup \mS_U$, which implies that $T^{-j}(B) \cap U \subset T^{-M}(B) \cap U$ and our setup is unchanged.  This concludes Case 2A.

\medskip\textit{Case 2B:  $U$ meets $\Ee_M'$.}  First note that $V$ is in the positive contracting direction with respect to $W$.  Let $n > M$ correspond to the tube that next intersects the open parallelogram $V$.  If $V$ meets $\Cc_n \cup \Cc_n'$,  then we are in Case 1 and thus finished.

Otherwise, $V$ must meet the new tube in the interior of $\Ee_n$ or the interior of $\Ee_n'$ or $V$ is completely contained in the new tube.  Now $\Ee_n \subset \Ee_M$.  Consider the case $V$ (and thus $U$) meets the interior of $\Ee_n$, then $U$ also meets the interior of $\Ee_M$, a contradiction.  Now consider the case that $V$ meets the interior of $\Ee_n'$.   Since we are allowed to move in the negative contracting direction from $\Ee_n'$ and remain in the new tube (and the new tube is much longer in this direction than $U$), the new tube meets $W \cup S_U$ and thus meets $y$, a contradiction.  Finally, consider the case that $V$ is completely contained in the new tube.  As in the previous case, we are allowed to move in the negative contracting direction from $\Ee_n'$ and remain in the new tube.  We meet $V$ and may move beyond $V$.  We may not, however, move beyond $S_U$ because, otherwise, $U$ would meet $y$, a contradiction.  Hence, the only possibility is that $U \cap \Ee_n = S_U$, which implies that $U$ meets the interior of $\Ee_M$, also a contradiction.  Consequently, for Case 2B, $V$ must meet $\Cc_n \cup \Cc_n'$, and we are in Case 1 and thus finished.  This concludes Case 2B.

\bigskip

Cases 2A and 2B now provide a recursive algorithm, which must terminate in a finite number of steps by entering Case 1.  If not, then the only way for the algorithm to continue is for Case 2A to be repeatedly used.  Let $\ell_U$ denote the length of the expanding edges of $U$.  Now at each step, the constructed $\widetilde{U}$ and $\widetilde{V}$ are open $T$-parallelograms both having expanding edges of length $\ell_U$.  Consequently, there exist an increasing sequence of times corresponding to tubes which are the first to intersect the sequence of $\widetilde{V}$ and the intersection will not contain the parts of the tubes that are the iterates of $\Cc$ or $\Cc'$.  Therefore, a slice from any tube in this sequence will have length larger than or equal to $\ell_U$.  But since this sequence is comprised of inverse iterates of $B$ under $T$, the expanding direction is contracting, and thus the length of the slices must become smaller than $\ell_U$, a contradiction.  This concludes Case 2 and proves the lemma.
\end{proof}

\section{A Baire category argument}\label{secBaireCatArg}

In this section, we prove Theorems~\ref{thmForDenNonDen} and~\ref{thmForDenNonDenAnosov} and their corollaries.  First, we require two lemmas (which, as an aside, also hold for complete orbits).  These will be widely useful in studying nondense orbits and will be of independent interest.

\begin{lemm}\label{lemmNonDenseisNowhereDense} Let $(X, {\mathcal B}, \mu)$ be a Borel probability space and $f: X \rightarrow X$ be an ergodic homeomorphism.  Then a nondense orbit is a nowhere dense orbit.
\end{lemm}
\begin{proof}
Assume not.  Let $x \in X$ have a nondense orbit that is not nowhere dense.  Thus $\overline{\Or_f^+(x)}$ contains a nonempty open set $U$.  By ergodicity, there is some point $y \in U$ with dense orbit.  Thus, \[X= \overline{\cup_{n=0}^\infty f^n(U)} \subset \overline{\Or_f^+(x)},\] a contradiction.   \end{proof}

\begin{lemm}\label{lemmNondenseOrbitisNowhereDense} Let $(X, {\mathcal B}, \mu)$ be a Borel probability space and $f: X \rightarrow X$ be an ergodic homeomorphism.  Let $n \in \NN$.  The point $x \in X$ has nondense orbit under $f \iff$ the point $x \in X$ has nondense orbit under $f^n$. \end{lemm}

\begin{proof}
The forward implication is obvious.  We prove the reverse implication. Assume it is false.  Note that
\[\overline{\Or_f^+ (x)} = \cup_{i=0}^{n-1} f^i(\overline{\Or_{f^n}^+(x)}).\]  By Lemma~\ref{lemmNonDenseisNowhereDense}, the right-hand side is a finite union of closed sets with empty interior and hence has an empty interior itself, a contradiction. 
\end{proof}

\subsection{Proofs for toral automorphisms}\label{secProofthmForDenNonDen}

Let $S: \TT^2 \rightarrow \TT^2$ be a hyperbolic toral automorphism.  Let $\Bb_S(z, \rho)$ denote the open $S$-parallelogram of diameter $\rho$ with equal length sides and with $z$ at the barycenter of the parallelogram.  For $S$, we define, as we did for $T$, the analogous notions of tubes and slices.  Since a hyperbolic toral automorphism has a local product structure, in any small enough open ball on $\TT^2$, line segments parallel to $E_S^-$ and $E_S^+$ are well-defined.

\begin{proof}[Proof of Theorem~\ref{thmForDenNonDen}]  Since the dimension of the torus is $2$, the eigenvalues of $T$ are real.  There are two cases. The first is that all the eigenvalues are positive.  
By Lemma~\ref{lemmForwardFracContainsContrSpace}, we have that an open $T$-parallelogram $U$ containing a point $y \in F_T$ also contains an open interval $I$ completely contained in $F_T$ and parallel to $E_T^-$.  Thicken this interval slightly to obtain an open $T$-parallelogram $V \subset U$.

Let $r \in \QQ^2 \cap \TT^2$ and $\ell$ be a natural number.  Consider the open parallelogram \[P:=\Bb_S(r, \ell^{-1}).\]  Since $r$ is rational, it is periodic under $S$.  Let $N$ be its period.  Let $\Cc_P$ denote the line segment through $r$ in $P$ contained in $E_S^-(r)$.  Since $S^{-N}$ fixes $r$, it fixes the leaf of the stable foliation through $r$, namely $E_S^-(r)$.  Moreover, an interval $D \subset E_S^-(r)$will be expanded by $S^{-N}$ and also contain $D$.  In particular, $\Cc_P \subset S^{-N} \Cc_P$.  Since $S^{-N}$ is ergodic, there exists $m \geq 0$ such that the $S^{-mN}$-tube for $P$ will meet $V$, and $V$ completely contains a closed slice of this tube (because slices under $S^{-1}$ contract).  Therefore, $S^{-mN}\Cc_P$ meets $V$.  Now this iterate of $\Cc_P$ is, locally in $V$, a line segment parallel to $E_S^-$, which by supposition is not parallel to $E_T^-$.  

Applying $S^{-N}$ more times will expand $\Cc_P$ (which stays on the same leaf) so that it meets all of the points of $V$ that are contained in a line segment parallel to $E^-_S$.  In particular, this iterate of $\Cc_P$ will intersect $I$ by incidence geometry at, say, a point $z'$ and this line segment is contained in $E^-_S(z')$.  And this intersection point $z'$, which belongs to $F_T$, will return to $P$ under enough iterates of $S$.

Thus, for all $r$ and $\ell$, the sets \[\bigcup_{n=0}^\infty S^{-n}\Bb_S(r, \ell^{-1}) \cap F_T\] are open dense subsets of the Baire space $F_T$.  Consequently, the Baire Category Theorem implies that the set \[A: = \bigcap_{r \in \QQ^2 \cap  \TT^2} \bigcap_{\ell = 1}^\infty \bigcup_{n=0}^\infty S^{-n} \Bb_S(r, \ell^{-1})\] restricted to $F_T$ is a dense $G_\delta$ subset of $F_T$.  

Now consider the parallelogram $B$ used in the construction of $F_T$.  Take a shrinking sequence of such parallelograms $\{B_j\}_{j=0}^\infty$ ($\boldsymbol{0}$ is a vertex in all of the parallelograms) and form the corresponding fractals $F_j := F_T(B_j)$.  Let $F := \cup_{j=0}^\infty F_j$; we note that $F$ is a subset of the set $ND(T)$ of points in the torus which have nondense forward orbits under $T$.  The set $F$ is also winning and hence dense. (The set $F$ is a superset of the winning set~\cite[Theorem~1.2]{BFK} of points whose forward orbits miss a neighborhood of $\boldsymbol{0}$).  The set $A$ is dense in all $F_j$.  

We now show that $A \cap F$ is dense.  Let $W$ be an open set of $\TT^2$.  Since $F$ is dense, $W$ contains a point $z \in F$.  But, $z \in F_j$ for some $j$.  Since $A$ restricted to $F_j$ is dense in $F_j$, there exists some $a \in A \cap F_j$ such that $a \in W$. Consequently, $A \cap F$ is a dense uncountable subset of the torus.

Finally, it is clear that $A$ is a subset of the set of points in the torus which have dense forward orbits under $S$.  This proves the theorem for $T$ with only positive eigenvalues.

For the general case of $T$ with real eigenvalues, we have that $T^2$ has positive eigenvalues.  By Lemma~\ref{lemmNondenseOrbitisNowhereDense}, the sets $ND(T)$ and $ND(T^2)$ are the same.  Consequently, $A \cap ND(T)$ is dense and uncountable.  This proves the theorem.
\end{proof}

\begin{proof}[Proof of Corollary~\ref{coroForDenNonDen}]  The proof is immediate from the proof of the theorem.\end{proof}

\begin{rema}\label{remNonCommResultStronger}
We may assert a little more (still assuming that $\dim(E_T^- \cap E_{S_n}^-) = 0$ and also adding in the additional assumption that all eigenvalues of $T$ are positive) than the conclusion of Theorem~\ref{thmForDenNonDen} or Corollary~\ref{coroForDenNonDen}.  Fix some open parallelogram $B$ as in the construction of the fractal $F_T$.  Then the set of points that miss $B$ under forward iterates of $T$ and that have dense forward orbits under all the $S_n$ is a dense $G_\delta$ subset of $F_T \subset \TT^2$.  Also, note that we may replace the role of $\boldsymbol{0}$ in the proof of the theorem with any rational point on the torus.  This is done by replacing $T$ with the power of $T$ that fixes the rational point and using Lemma~\ref{lemmNondenseOrbitisNowhereDense}.
\end{rema}

\subsection{Proofs for toral diffeomorphisms}\label{secProofthmForDenNonDenAnosov}

The global classification theorem for Anosov diffeomorphisms on tori (see~\cite[Theorem~18.6.1]{HKIntroMD} for example) is the following:

\begin{theo}\label{thmGlobClassAnoDiffeoTori}  Let $d \geq 2$.  Every Anosov diffeomorphism of $\TT^d$ is topologically conjugate to a linear hyperbolic toral automorphism.
 
\end{theo}

Let $\widetilde{T}:\TT^2 \rightarrow \TT^2$ be a $C^2$-Anosov diffeomorphism.  The theorem implies that there exists the following commutative diagram: \[\begin{CD}
\TT^2     @>\widetilde{T}>>  \TT^2\\
@VVhV        @VVhV\\
\TT^2      @>T>>  \TT^2 
\end{CD}\]  Here $T$ is a (linear) hyperbolic toral automorphism and $h$ is a homeomorphism.

Let $B^n \subset \RR^n$ denote the unit ball.  Suppose $E$ is a partition of a $d$-dimensional $C^1$-manifold $M$ into injectively immersed $n$-dimensional $C^1$-submanifolds and, for any $x \in M$, let $E(x)$ denote the submanifold containing $x$.  Recall that a \textit{foliation} (or, more precisely, \textit{$C^1$-foliation}) $E$ of $M$ is such a partition for which every $x \in M$ has an open neighborhood $U$ and a homeomorphism $\varphi: B^n \times B^{d-n} \rightarrow U$ such that\footnote{One can weaken the regularity condition to define other notions of foliation, but we have no need to do this as all the foliations we consider are $C^1$ by Theorem~\ref{thmC2AnosovC1Fol} and its remark.} \begin{itemize}
\item for each $z \in B^{d-n}$, the set $\varphi(B^n \times \{z\})$ is the connected component of $E(\varphi(\boldsymbol{0}, z)) \cap U$ containing $\varphi(\boldsymbol{0},z)$ and \item $\varphi(\cdot, z)$ is $C^1$ and depends continuously on $z$ in the $C^1$-topology.
\end{itemize}
  The sets $E_U(x):=\varphi(B^n \times \{z\})$ are called \textit{plaques} (or \textit{local leaves}) and the submanifolds $E(x)$ are called \textit{leaves}.

A classical result is the following~\cite[Corollary~4]{HP}:

\begin{theo}\label{thmC2AnosovC1Fol} Let $f$ be a $C^2$-Anosov diffeomorphism of a compact manifold M. If the stable manifolds have codimension 1, they form a $C^1$-foliation of M.
 
\end{theo}
\begin{rema}
As the roles of the stable and unstable manifolds switch when $f$ is replaced by $f^{-1}$, the theorem also applies to unstable manifolds.
\end{rema}

\noindent Note that Theorem~\ref{thmC2AnosovC1Fol} and its remark apply to $\widetilde{T}$ because both the stable and unstable manifolds are codimension 1.  Therefore, we refer to the collection of stable manifolds as the \textit{stable foliation} and the collection of unstable manifolds as the \textit{unstable foliation}.  Under topological conjugacy, the leaves of the stable foliation are mapped bijectively into the leaves of the stable foliation and the leaves of the unstable foliation are mapped bijectively into the leaves of the unstable foliation, which, for our case on $\TT^2$, is the following: \begin{align}\label{eqnAnosovLinearFoliations} E^-_{\widetilde{T}} (h^{-1} x ) = h^{-1}E^-_T(x) \quad \textrm{ and } \quad E^+_{\widetilde{T}} (h^{-1} x ) = h^{-1}E^+_T(x).
  \end{align}
  
We define a \textit{piece of space through a stable manifold} to be an intersection of a plaque of the stable foliation with an open neighborhood of $\TT^2$.  Since plaques of the stable foliation for an linear hyperbolic toral automorphism are line segments parallel to the stable manifold, this notion generalizes the notion, defined in Section~\ref{secGeoFracforT}, of a piece of space parallel to the stable manifold.  Moreover, since a tube for a $T$-parallelogram (see Section~\ref{secGeoFracforT}) is partitioned by line segments parallel to the stable manifold for $T$, its image under $h^{-1}$ is partitioned by plaques of the stable foliation for $\widetilde{T}$ (and also partitioned by plaques for the unstable foliation for $\widetilde{T}$); we refer to these images as tubes for $\widetilde{T}$.  Slices, which are plaques in the unstable foliation, are defined analogously.

Let $B$ and $F_T(B)$ be defined as in Section~\ref{secGeoFracforT} for the linear hyperbolic toral automorphism $T$.  Let \[F(B):= F_T(B) \quad \textrm{ and } \quad \widetilde{F}( h^{-1} B):= F_{\widetilde{T}}(h^{-1} B) = \TT^2 \backslash \bigcup_{n=0}^\infty \widetilde{T}^{-n} h^{-1} B\] and note that \begin{align}\label{eqnAnosovLinearFracs} \widetilde{F}( h^{-1} B) = h^{-1} F(B).
  \end{align}

\noindent We now have
\begin{lemm}\label{lemmForwardFracContainsContrSpaceAnosov} Let the eigenvalues of $T$ be positive real numbers.  Any open subset of $\widetilde{F}( h^{-1} B)$ contains a piece of space through a stable manifold for $\widetilde{T}$. 
\end{lemm}

\begin{proof}
 Apply (\ref{eqnAnosovLinearFoliations}, \ref{eqnAnosovLinearFracs}) to Lemma~\ref{lemmForwardFracContainsContrSpace}.
\end{proof}

Let $\widetilde{S}:\TT^2 \rightarrow \TT^2$ be a $C^2$-Anosov diffeomorphism.  Theorem~\ref{thmGlobClassAnoDiffeoTori} also implies that there exists the following commutative diagram: \[\begin{CD}
\TT^2     @>\widetilde{S}>>  \TT^2\\
@VVgV        @VVgV\\
\TT^2      @>S>>  \TT^2 
\end{CD}\]  Here $S$ is a (linear) hyperbolic toral automorphism and $g$ is a homeomorphism.   Let $\Bb_S(z, \rho)$ be the open $S$-parallelogram defined in Section~\ref{secProofthmForDenNonDen}.

Recall that, for a $C^1$-manifold $M$, two $C^1$-submanifolds $N_1$ and $N_2$ of complementary dimensions \textit{intersect transversely} at a point $x \in M$ if $T_xN_1 \oplus T_xN_2 = T_x M$.  For a transverse intersection, we have adapted coordinates (see~\cite[Lemma~A.3.17]{HKIntroMD} for example):
\begin{lemm}\label{lemmAdaptedCoord}  There exists a neighborhood $U \subset M$ of $x$ and coordinates $(x_1, \cdots, x_d)$ on $U$ such that in these coordinates \[N_1 \cap U = \{(x_1, \cdots, x_d): x_{n_+1} = \cdots = x_d =0\}\] and \[N_2 \cap U = \{(x_1, \cdots, x_d): x_{1} = \cdots = x_n =0\}.\]
  
\end{lemm}
\noindent A plaque of the stable foliation of $\widetilde{T}$ and a plaque of the stable foliation of $\widetilde{S}$ are both one-dimensional and, hence, should they intersect, then, by Theorem~\ref{thmC2AnosovC1Fol}, they intersect transversely at a point in $\TT^2$ if and only if their tangent lines at that point do not coincide.

\begin{proof}[Proof of Theorem~\ref{thmForDenNonDenAnosov}]  Using Lemma~\ref{lemmNondenseOrbitisNowhereDense}, we may assume, without loss of generality, that $T$ has positive eigenvalues.  By Lemma~\ref{lemmForwardFracContainsContrSpaceAnosov}, we have that an open $\widetilde{T}$-tube $U$ containing a point $y \in \widetilde{F}(h^{-1}B)$ also contains a piece of space $I$ through a stable manifold for $\widetilde{T}$.  The set $I$ is a plaque and, using Theorem~\ref{thmC2AnosovC1Fol}, is a path-connected piece of a $C^1$-curve.  Let $\widetilde{I} \subset I$ be a path-connected subset containing both of its endpoints.  Since the leaves of the stable foliation partition $\TT^2$, we may thicken this closed piece of plaque $\widetilde{I}$ into a closed $\widetilde{T}$-tube $W \subset U$. Let $V$ be the interior of $W$.  It is an open $\widetilde{T}$-tube.  Note that $V \cap \widetilde{I}$ is a plaque and a subset of $I$.

Let $r \in \QQ^2 \cap \TT^2$ and $\ell\geq 4$ be a natural number.  Consider the open $\widetilde{S}$-tube \[P:=g^{-1}\Bb_S(r, \ell^{-1}),\] which is partitioned by plaques of the stable foliation for $\widetilde{S}$.  In particular, it contains the plaque \[\Cc_P:= E^-_{\widetilde{S}}(g^{-1} r) \cap P.\]  Since $r$ is rational, $g^{-1} r$ is periodic under $\widetilde{S}$.  Let $N$ be its period.  Since $\widetilde{S}^{-N}$ fixes $g^{-1}r$, it fixes $E_{\widetilde{S}}^-(g^{-1}r)$ and, moreover, $\Cc_P \subset \widetilde{S}^{-N} \Cc_P$, as the inverse is expanding on the leaves of the stable foliation.  Since $\widetilde{S}^{-N}$ is topologically mixing (\cite[Proposition~18.6.5]{HKIntroMD}), there exists $m \geq 0$ such that the $\widetilde{S}^{-mN}$-tube for $P$ will meet $V$, and $V$ completely contains a closed slice of this tube (because slices under $\widetilde{S}^{-1}$ contract).  Therefore, $\widetilde{S}^{-mN}\Cc_P$ meets $V$ and $J:=\widetilde{S}^{-mN}\Cc_P \cap V$ is (part of) a plaque of stable foliation for $\widetilde{S}$.  

We claim that the leaf $E^-_{\widetilde{S}}(g^{-1} r)$ intersects $\widetilde{I}$.  Assume the claim is false.  We know that the plaque in $U$ of the leaf $E^-_{\widetilde{S}}(g^{-1} r)$ through the closed tube $W$ is a path-connected piece of a $C^1$-curve containing both of its endpoints, which we denote by $\widetilde{J}$ and which is nonempty (because it contains $J$).  Since $W$ is a closed $\widetilde{T}$-tube, it is partitioned by pieces of plaque from the stable foliation for $\widetilde{T}$ and each piece of plaque $I_\alpha$ is a path-connected piece of a $C^1$-curve containing both of its endpoints.  By assumption, $\widetilde{I}$ does not meet $\widetilde{J}$.  Therefore, we have two nonempty equivalence classes on the collection of pieces of plaques $\{I_\alpha\}$ that forms the partition for $W$ :  those that meet $\widetilde{J}$ and those that do not.  The same is true for a slightly smaller closed $\widetilde{T}$-tube $\widetilde{W} \subset W$.

By the continuity of the foliation on the compact set $W$, we can pick a plaque from one equivalence class and another plaque from the other equivalence class that are arbitrarily close (for, otherwise, the two classes form a separation, contradicting the fact that $W$ is connected).  For each piece $I_\alpha$ that does not meet $\widetilde{J}$, find a point $x_\alpha \in I_\alpha$ closest to $\widetilde{J}$.  Such a point exists because $I_\alpha$ and $\widetilde{J}$ are both compact.  We observe that there exists a sequence of these points $\{x_j\} \subset \{x_\alpha\}$ such that $\dist(x_j, \widetilde{J}) \rightarrow 0$ as $j \rightarrow \infty$.  If this were false, then every point in every $I_\alpha$ not meeting $\widetilde{J}$ would be a positive distance $\ell >0$ from $\widetilde{J}$.  Pick a plaque $I_\beta$ meeting $\widetilde{J}$, which is arbitrarily close to a plaque $I_\gamma$ not meeting $\widetilde{J}$.  This includes an intersection point of $I_\beta$ and $\widetilde{J}$, which is a contradiction.  Note that our observation also holds for $\widetilde{W}$.

Consider such a sequence of ${x_j}$ on $\widetilde{W}$ for which $\dist(x_j, \widetilde{J}) \rightarrow 0$ as $j \rightarrow \infty$.  As $\widetilde{W} $ is compact, a subsequence converges and, thus, there exists a limit point $x_\infty$ for which $\dist(x_\infty, \widetilde{J})=0$ and $x_\infty \in \widetilde{W} \subset V$.  The point $x_\infty$ must lie on some $I_\kappa$ which meets $V$ and, since $\widetilde{J}$ is compact, also on $\widetilde{J}$.  By assumption, the intersection of $I_\kappa$ with  $\widetilde{J}$ is transverse.  Using Lemma~\ref{lemmAdaptedCoord}, there exists a small open $\widetilde{T}$-tube $\widetilde{V} \subset W$ containing $x_\infty$ for which \[\widetilde{J} = \{(x_1, x_2) : x_1 =0\} \quad \textrm{ and } \quad I_\kappa= \{(x_1, x_2) : x_2 =0\}.\]  Any plaque $I_\alpha \cap \widetilde{V}$ (but the plaque $I_\kappa \cap\widetilde{V}$) from the stable foliation for $\widetilde{T}$ does not meet $I_\kappa$ and hence has nonzero $x_2$-coordinate, but remains a $C^1$-curve.  Since $\overline{\widetilde{V}}$ is compact and the stable foliation for $\widetilde{T}$ is $C^1$, the derivative of each $I_\alpha$ over all points in $\widetilde{V}$ is close to each other and, in particular, in adapted coordinates, this implies that the derivative of $I_\alpha$ at each point in $\widetilde{V}$ is close to the derivative of $I_\kappa$ at $x_\infty$, which is zero.  Consequently, the tangent line at $x_\infty$ of $I_\kappa$ is horizontal and the tangent lines of any point in $\widetilde{V}$ through its $I_\alpha$ has absolute value of their slope bounded by some small $1/4>\delta>0$ (perhaps after choosing $\widetilde{V}$ to be smaller). 

Let $B_R$ denote a closed $\|\cdot\|_\infty$-ball $B_R$ of radius $R$ around $x_\infty$ (which is the origin in the adapted coordinates).
In the adapted coordinates, pick a $B_R$ contained in $\widetilde{V}$. By construction, there exists a plaque $I_\omega \cap V$ of the stable foliation for $\widetilde{T}$ arbitrarily close to $I_\kappa$ such that $I_\omega \cap V \cap \widetilde{J} = \emptyset$.  Pick a point $y$ of $I_\omega \cap V$ within $B_{R/4}$.  Now $y$ lies in a proper quadrant of $B_R$ because it cannot meet $\widetilde{J}$, which in adapted coordinates is the vertical axis, and it cannot meet $I_\kappa$ because $y$ does not lie on the plaque $I_\kappa$.  For exactly the same reason, every point in $I_\omega \cap V \cap B_R$ lies in the same proper quadrant.  

Now $I_\omega \backslash \widetilde{V}$ is two connected components.  Picking a point on each component and considering the arc of $I_\omega$ between that point and $y$ implies that the plaque $I_\omega$ must intersect the boundary square of $B_R$ in at least 2 points.  We assert that one of these intersection points must lie on the line segment parallel to $I_\kappa$ in the adapted coordinates.  Assume that the assertion were false.  Then there are at least two intersection points with the line segment of the boundary of $B_R$ perpendicular to $I_\kappa$ in the adapted coordinates.  The mean value theorem implies that $I_\omega$ has a vertical tangent line at a point in $\widetilde{V}$, namely the absolute value of the slope is unbounded, which is a contradiction that proves our assertion.

Let $z$ be the intersection point with the line segment parallel to $I_\kappa$.  Then the absolute value of the slope of the line through $y$ and $z$ is greater than or equal to $\frac{3R/4}{R}$.  Since both $y$ and $z$ lie on $I_\omega$, the mean value theorem implies that there exists a point $p \in I_\omega \cap \widetilde{V}$ whose tangent line has absolute value of slope greater than or equal to $3/4$, a contradiction.  This shows that the leaf $E^-_{\widetilde{S}}(g^{-1} r)$ intersects $\widetilde{I}$ as claimed.  Call this intersection point $z'$.    And, this intersection point $z'$, which belongs to $\widetilde{F}(h^{-1} B)$, will return to $P$ under enough iterates of $\widetilde{S}$.  

The remainder of the proof is the same as that of Proof of Theorem~\ref{thmForDenNonDen}.  We note that the set of points that miss a family of shrinking neighborhoods is uncountable and dense because its image under the conjugacy is winning (as this image is the analogous set for a linear hyperbolic toral automorphism) and homeomorphisms preserve cardinality and density.  (For certain Anosov diffeomorphisms, it is also known that their nondense sets are winning~\cite{Ts15a}, but this fact is not necessary for our proof.)  This proves the theorem. \end{proof}

\begin{proof}[Proof of Corollary~\ref{coroForDenNonDenAnosov}]  The proof is immediate from the proof of the theorem.\end{proof}

\end{document}